\documentclass[10pt,leqno,letterpaper]{article}
\usepackage[T1]{fontenc}
\usepackage[utf8]{inputenc}
\usepackage{amsmath}
\usepackage{amsfonts}
\usepackage{amssymb}
\usepackage{graphicx}
\usepackage{verbatim}
\usepackage{mathrsfs}
\usepackage{upref,amsthm,amsxtra,exscale}
\usepackage{nameref}
\usepackage{varioref}
\usepackage[colorlinks=true,urlcolor=blue,
citecolor=red,linkcolor=blue,linktocpage,pdfpagelabels,
bookmarksnumbered,bookmarksopen]{hyperref}
\usepackage{cleveref}
\usepackage{cite}
\usepackage{fullpage}
\usepackage[dvipsnames]{xcolor}

\newtheorem{theorem}{Theorem}[section]

\newtheorem{remark}[theorem]{Remark}

\newtheorem{lemma}[theorem]{Lemma}

\newtheorem{question}[theorem]{Question}

\numberwithin{equation}{section}

\def\r{\mathbb{R}}
\def\rn{\mathbb{R}^N}
\def\z{\mathbb{Z}}
\def\b{\mathbb{B}}

\def\zl{\z/\ell\z}
\def\s1{\mathbb{S}^1}
\def\n{\mathbb{N}}
\def\cc{\mathbb{C}}
\def\eps{\varepsilon}

\def\io{\int_{\Omega}}
\def\irn{\int_{\r^N}}
\def\vp{\varphi}

\def\vr{\varrho}
\def\o{\Omega}

\def\bf{\boldsymbol}

\def\cC{\mathcal{C}}
\def\cD{\mathcal{D}}

\def\cH^G{\mathcal{H}}

\def\cJ{\mathcal{J}}

\def\cM{\mathcal{M}}
\def\cN{\mathcal{N}}

\def\supp{\text{supp}}
\def\bar{\overline}
\def\what{\widehat}
\def\d{\,\mathrm{d}}

\def\dist{\mathrm{dist}}

\def\div{\text{div}}
\def\p*{p^*}

\author{Mónica Clapp and Víctor A. Vicente-Benítez\footnote{V. A. Vicente-Ben\'itez was supported by CONACYT (Mexico) through the research program {\it Postdoctoral Stays for Mexico 2023(1)}}}
\title{Entire solutions to a quasilinear purely critical competitive system}
\date{\today}

\begin{document}
	\maketitle
	
\begin{abstract}
We establish the existence of a fully nontrivial solution with nonnegative components for a weakly coupled competitive system for the $p$-Laplacian in $\rn$ whose nonlinear terms are purely critical.

We also show that the purely critical equation for the $p$-Laplacian in $\rn$ has infinitely many nodal solutions.
\smallskip

\smallskip

\end{abstract}
        
\section{Introduction}
        
We consider the weakly coupled competitive critical quasilinear system 
\begin{equation}\label{eq:rn_general_system}
\begin{cases}
-\Delta_p u_i = \sum\limits_{j=1}^\ell\beta_{ij}|u_j|^{\frac{p^*}{2}}|u_i|^{\frac{p^*}{2}-2}u_i, \\
u_i\in D^{1,p}(\rn),\qquad i=1,\ldots,\ell,
\end{cases}
\end{equation}
where $\Delta_pu:= \div\left(|\nabla u|^{p-2}\nabla u\right)$ is the $p$-Laplacian of $u$, $\frac{2N}{N+2}<p<N$, $\beta_{ii}>0$, $\beta_{ij}<0$ if $i\neq j$, $p^*:=\frac{Np}{N-p}$ is the critical Sobolev exponent, $\ell\geq 2$, and 
$$D^{1,p}(\rn):= \{ u\in L^{p^*}(\rn):\partial_iu\in L^p(\rn), \, i=1,\dots, N\}.$$

We also consider the critical quasilinear equation
\begin{equation}\label{eq:equation}
-\Delta_p w=|w|^{p^*-2}w,\qquad w\in D^{1,p}(\rn).
\end{equation}
For $p=2$ this equation is conformally equivalent to the Yamabe equation on the standard $N$-sphere, by means of the stereographic projection.

Damascelli, Merchán, Montoro, and Sciunzi \cite{dmms}, Sciunzi \cite{s} and Vétois \cite{v} showed that, for any $p\in(1,N)$, the problem \eqref{eq:equation} has a unique positive solution up to translations and dilations, which is the least energy solution. The existence of nodal solutions was established in \cite{cl}, where it is shown that \eqref{eq:equation} has a finite number of nonradial solutions that change sign, which depends on the dimension $N$.

The existence of infinitely many nodal solutions to \eqref{eq:equation} is known only for $p=2$. It was first proved by W. Ding in \cite{di} using the invariance of the problem under conformal transformations. Del Pino, Musso, Pacard and Pistoia exhibited in \cite{dmpp} another family of sign-changing solutions, using the Lyapunov–Schmidt reduction method. None of these methods apply to the quasilinear case. As shown in \cite{l}, the $p$-Laplacian is not invariant under conformal transformations when $p\neq 2$. On the other hand, the Lyapunov–Schmidt procedure cannot be used because the linearized operator for the $p$-Laplacian is not well understood.

In \cite{c} another type of nodal solutions for $p=2$ was obtained, using only the invariance of the problem under linear isometries. Inspired by this approach we prove the following result.
        
\begin{theorem}\label{thm:multiplicity}
Let $N\geq 4$. Then, for any $p\in(1,N)$, the quasilinear problem \eqref{eq:equation} has an infinite sequence of nonradial sign-changing solutions that are not equivalent under translation and dilation.
\end{theorem}

Next, we look at the system. Motivated by their physical applications, weakly coupled elliptic systems have received much attention in recent years, specially in the semilinear case ($p=2$). 

Note that every solution $w$ to \eqref{eq:equation} gives rise to a solution of the system \eqref{eq:rn_general_system} whose $i$-th component is $\beta_{ii}^{(p-N)/p^2}w$ and all other components are trivial. The interesting question is whether \eqref{eq:rn_general_system} has a \emph{fully nontrivial} solution, i.e., a solution $\bf u=(u_1,\ldots,u_\ell)$  such that all components $u_i$ are nontrivial. This is not an easy matter. 

First, we observe that the system cannot be solved by minimization.
        
\begin{theorem}\label{thm:nonexistence}
The system \eqref{eq:rn_general_system} does not have a least energy fully nontrivial solution. 
\end{theorem}
        
This is a consequence of the invariance of the system under dilations and the strong maximum principle for the $p$-Laplacian. In fact, the same is true for any domain $\o$ in $\rn$; see Theorem \ref{thm:o_nonexistence} below. For a system of two equations this was proved by Chen and Zou in \cite[Theorem 1.6]{cz} when $p=2$, and by Guo, Perera and Zou in \cite[Theorem 1.1]{gpz} for general $p$.
        
The question now is whether there are higher energy fully nontrivial solutions. Some existence results are known for $p=2$. Guo, Li and Wei \cite{glw} studied the system \eqref{eq:rn_general_system} of two equations in dimension $N=3$ and established the existence of nonnegative solutions with $k$ peaks for sufficiently large $k$ using the Lyapunov-Schmidt reduction method. The existence of infinitely many solutions to \eqref{eq:rn_general_system} was established in \cite{cp1,cf,csz} by exploiting the fact that the system is invariant under conformal transformations of $\rn$ when $p=2$. As we mentioned above, none of these methods apply if $p\neq 2$. Furthermore, even for $p=2$, the Lyapunov-Schmidt procedure can only be used for small dimensions where the variational functional has enough regularity.

Cooperative quasilinear systems in $\rn$ with critical nonlinear terms have been considered in \cite{az,zz,gpz}, and results in bounded domains are given, for example, in \cite{ccmv, ds, dx, mp}. But we know of no existence result for the \emph{competitive} critical system \eqref{eq:rn_general_system} with $p\neq 2$. 
        
Our goal is to use symmetries to produce solutions. Therefore, we consider the symmetric system
\begin{equation}\label{eq:rn_symmetric_system}
\begin{cases}
-\Delta_p u_i = |u_i|^{p^*-2}u_i+\sum\limits_{\substack{j=1\\j\neq i}}^\ell\beta|u_j|^{\frac{p^*}{2}}|u_i|^{\frac{p^*}{2}-2}u_i, \\
u_i\in D^{1,p}(\rn),\qquad i=1,\ldots,\ell,
\end{cases}
\end{equation}
where $\frac{2N}{N+2}<p<N$, $N\geq 4$, $\beta<0$, $p^*:=\frac{Np}{N-p}$ and $\ell\geq 2$.
        
We now describe the symmetries we consider. We write $\rn\equiv\cc^2\times\r^{N-4}$ and a point in $\rn$ as $x=(z,y)$ with $z=(z_1,z_2)\in\cc\times\cc$ and $y\in\r^{N-4}$. Let $\tau:\cc^2\to\cc^2$ be given by
\begin{equation*}
\tau(z_1,z_2):=(-\overline{z}_2,\overline{z}_1),
\end{equation*}
where $\overline{z}$ is the complex conjugate of $z$. We define $G$ to be the group whose elements are the unit complex numbers acting on $\rn$ by
\begin{equation}\label{eq:G}
gx:=(gz_1,\overline{g}z_2,y),\qquad\text{for \ }g\in G, \ \ x=(z_1,z_2,y)\in\cc\times\cc\times\r^{N-4},
\end{equation}
and take $\vr_\ell:\rn\to\rn$ to be
\begin{equation}\label{eq:rho}
\vr_\ell(z,y):=\left(\Big(\cos\frac{\pi}{\ell}\Big)z+\Big(\sin\frac{\pi}{\ell}\Big)\tau z,\,y\right)\qquad\text{for \ }(z,y)\in\cc^2\times\r^{N-4}.
\end{equation}
We shall look for solutions $\bf u=(u_1,\ldots,u_\ell)$ that satisfy
\begin{itemize}
\item[$(\bf{P_1})$] $u_j$ is $G$-invariant for every $j=1,\ldots,\ell$, i.e., $u_j(gx)=u_j(x)$ for any $g\in G$ and $x\in\rn$.
\item[$(\bf{P_2})$] $u_{j+1}=u_j\circ\vr_\ell$ \ for every $j=1,\ldots,\ell$, \ where \ $u_{\ell+1}:=u_1$.
\end{itemize}
They are called \emph{pinwheel solutions}. We prove the following result.
        
        \begin{theorem}\label{thm:existence}
        The system \eqref{eq:rn_symmetric_system} has a nontrivial pinwheel solution which has least energy among all nontrivial pinwheel solutions and whose components are nonnegative. 
        \end{theorem}
        
        For $p=2$ this result was recently proved in \cite[Theorem 1.1]{cfs}. The symmetries used there are the same as those just described and the basic strategy of the proof is formally the same. It is based on a careful analysis of the behavior of minimizing sequences for the variational functional associated with the system in the unit ball; see Theorem \ref{thm:minimizing}. But here we find a substancial difference from the semilinear case. When $p=2$ one can use a unique continuation principle for systems, which was established in \cite{chs}, to rule out the existence of nontrivial pinwheel solutions in a ball and in a half-space, leaving as the only option the existence of a pinwheel solution in $\rn$; see Remark \ref{rem:unique continuation}. Unfortunately, a unique continuation principle is not available for $p\neq 2$, not even for the single equation \eqref{eq:equation}; see \cite{gm}. We avoid this difficulty by showing that a least energy pinwheel solution in a ball or a half-space gives rise to a least energy pinwheel solution in $\rn$.
        
Theorem \ref{thm:existence} appears to be the first existence result for the purely critical competitive quasilinear system \eqref{eq:rn_general_system}.

We note that pinwheel solutions for subcritical Schrödinger systems have been exhibited in \cite{cp2, cssv, pv, pw, ww}.
        
We close this introduction with an open question. It is well known that the $p$-Laplacian satisfies a strong maximum principle; see \cite[Theorem 2.5.1]{ps}. Since we are assuming that $\beta_{ij}<0$ if $i\neq j$, this principle cannot be applied to each equation of the system \eqref{eq:rn_general_system}. But perhaps a weak unique continuation property for systems, similar to the one obtained in \cite{chs}, holds true for solutions with \emph{nonnegative} components. This would allow us to exclude the existence of a least energy pinwheel solution in a ball and in a half-space, as happens in the semilinear case. Our precise question is the following:
        
        \begin{question}
        Is it true that, if $\bf u=(u_1,\ldots,u_\ell)$ is a solution to the system \eqref{eq:rn_general_system} with $u_i\geq 0$ for all $i$ and $\bf u\equiv \bf 0$ in some open subset of $\rn$, then $\bf u\equiv \bf 0$ in all of $\rn$?
        \end{question}
        
The paper is organized as follows. In Section \ref{sec:nonexistence} we prove a general nonexistence result that includes Theorem \ref{thm:nonexistence}. The variational setting for pinwheel solutions is discussed in Section \ref{sec:pinwheel}. In Section \ref{sec:existence} we analize the behavior of minimizing sequences for the symmetric system in the unit ball and we prove Theorem \ref{thm:existence}. Section \ref{sec:nodal} is devoted to the proof of Theorem \ref{thm:multiplicity}.

	\section{Nonexistence of ground state solutions}
	\label{sec:nonexistence}
	
	Consider the system
	\begin{equation}\label{eq:o_system}
		\begin{cases}
			-\Delta_p u_i = \sum\limits_{j=1}^\ell\beta_{ij}|u_j|^{\frac{\p*}{2}}|u_i|^{\frac{\p*}{2}-2}u_i, \\
			u_i\in D^{1,p}_0(\o),\qquad i=1,\ldots,\ell,
		\end{cases}
	\end{equation}
	where $\o$ is a domain in $\rn$, $\frac{2N}{N+2}<p<N$, $\beta_{ii}>0$, $\beta_{ij}<0$ if $i\neq j$ and $p^*:=\frac{Np}{N-p}$. 
	
	The space $D_0^{1,p}(\o)$ is the closure of $\cC^\infty_c(\o)$ in the Banach space 
	$$D^{1,p}(\rn):= \{ u\in L^{\p*}(\rn): \partial_iu\in L^p(\rn), \, i=1,\dots, N\}$$ 
	equipped with its standard norm 
	$$\|u\|_p:= \left(\int_{\rn}|\nabla u(x)|^p\d x\right)^\frac{1}{p}.$$
We write $|\cdot|_q$ for the usual norm of $L^q(\rn)$, $q\geq 1$, and denote by
	\begin{equation*}
		S_p:= \inf_{\overset{u\in D^{1,p}(\rn)}{u\neq 0}}\frac{\|u\|^p_p}{|u|_{p^*}^p}
	\end{equation*}
the best constant for the Sobolev embedding $D^{1,p}(\rn)\hookrightarrow L^{p^*}(\rn)$.
	
	Consider the $\ell$-fold cartesian product $(D^{1,p}_0(\o))^{\ell}$ with the norm
	$$\|\bf u\|_p:=\left(\sum_{i=1}^\ell\|u_i\|_p^p\right)^\frac{1}{p},\qquad \bf u=(u_1,\dots,u_{\ell}),$$
and the functional $\cJ :(D^{1,p}_0(\o))^{\ell}\to\r$ given by
	$$\cJ(\bf u):=  \frac{1}{p}\sum_{i=1}^\ell\|u_i\|_p^p -\frac{1}{p^*}\sum_{i,j=1}^\ell\beta_{ij}\irn |u_j|^{\frac{p^*}{2}}|u_i|^{\frac{p^*}{2}}.$$
Since, by assumption, $p^*>2$, this is a $\cC^1$-functional.  Its $i$-th partial derivative at each $\bf u=(u_1,\dots,u_{\ell})\in (D^{1,p}_0(\o))^{\ell}$ is given by
		\begin{equation*}
			\partial_i\mathcal{J}(\bf u)v=\int_{\rn}|\nabla u_i|^{p-2}\nabla u_i\cdot\nabla v-\sum_{j=1}^\ell\beta_{ij}\irn|u_j|^{\frac{\p*}{2}}|u_i|^{\frac{p^*}{2}-2}u_iv\qquad\text{for all \ }v\in D^{1,p}_0(\o).
		\end{equation*}
The solutions $\bf u=(u_1,\dots, u_{\ell})$ of the system \eqref{eq:o_system} are the critical points of $\cJ$. The fully nontrivial ones, i.e., those whose components are nontrivial, belong to the Nehari-type set
\begin{equation} \label{M(o)}
\cM(\o):=\{\bf u\in(D^{1,p}_0(\o))^\ell:u_i\neq 0 \text{ \ and \ } \partial_i\mathcal{J}(\bf u)u_i=0\text{ \ for all \ } i=1,\ldots,N\}.
\end{equation}
Define
	$$\mu(\o):=\inf_{\bf u\in\cM(\o)}\cJ(u).$$
A function $\bf u\in\cM(\o)$ such that $\cJ(u)=\mu(\o)$ is called a \emph{least energy fully nontrivial solution}.

In order to estimate $\mu(\o)$, for each $i=1,\ldots,\ell$ we consider the critical quasilinear equation
	\begin{equation}\label{eq:singleequation}
		-\Delta_p w = \beta_{ii}|w|^{p^*-2}w, \qquad w\in D^{1,p}_0(\o),
	\end{equation}
The nontrivial solutions of \eqref{eq:singleequation} are the critical points of the restriction of the functional $J_i:D_0^{1,p}(\o)\to\r$, given by
	$$J_i(w):=\frac{1}{p}\|w\|_p^p -\frac{\beta_{ii}}{p^*}|w|^{p^*}_{p^*},$$
to the Nehari manifold 
	\begin{align*}
		M_i(\o):=\{w\in D^{1,p}_0(\o):w\neq 0 \text{ \ and \ } \|w\|_p^p =\beta_{ii}|w|^{p^*}_{p^*}\}.
	\end{align*}
	Using the invariance under dilations, one easily sees that
	\begin{equation}\label{eq:m_i}
		m_i:=\inf_{w\in M_i(\o)}J_i(w)=\frac{1}{N}\beta_{ii}^{-\frac{N-p}{p}}S_p^\frac{N}{p}.
	\end{equation}
Thus, $m_i$ does not depend on $\o$. 

It is well known that $m_i$ is not attained if $\rn\smallsetminus\o$ has nonempty interior. A stronger result is true for the system.
	
	\begin{theorem}\label{thm:o_nonexistence}
		For any domain $\o$ in $\rn$,
		$$\mu(\o)=\sum_{i=1}^\ell m_i,$$
		and $\mu(\o)$ is not attained.
	\end{theorem}
	
	\begin{proof}
Let $\bf u=(u_1,\dots, u_{\ell})\in \cM(\o)$. Then $u_i\neq 0$ and
		\[
		\|u_i\|_p^p=\beta_{ii}|u_i|_{p^*}^{p^*}+\sum_{\overset{j=1}{j\neq i}}^{\ell}\beta_{ij}\irn|u_j|^{\frac{p^*}{2}}|u_i|^{\frac{p^*}{2}}\leq \beta_{ii}|u_i|_{p^*}^{p^*}.
		\]
Setting $t_i:=\Big(\frac{\|u_i\|_p^p}{\beta_{ii}|u_i|_{p^*}^{p^*}}\Big)^{1/(p^*-p)}$ we have that $t_i\in(0,1]$ and $t_iu_i\in M_i(\o)$. Therefore,
\begin{equation} \label{eq:nonexistence}
\sum_{i=1}^\ell m_i\leq \frac{1}{N}\sum_{i=1}^\ell\|t_iu_i\|_p^p\leq \frac{1}{N}\sum_{i=1}^\ell\|u_i\|_p^p=\cJ(\bf u)\quad\text{for every \ }\bf u\in \cM(\o).
\end{equation}
As a consequence, $\sum_{i=1}^\ell m_i\leq \mu(\Omega).$

To prove the opposite inequality, for each $i=1,\ldots,\ell$ we take a sequence $(\vp_{i,k})$ in $\cC_c^{\infty}(\o)\cap M_i(\o)$ such that $J_i(\vp_{i,k})\to m_i$ as $k\to\infty$. We choose $\ell$ distinct points $x_1,\dots, x_{\ell}\in \o$ and $r>0$ such that $B_r(x_i)\subset \Omega$ and $B_r(x_i)\cap B_r(x_j)=\emptyset$ if $i\neq j$. Since $\vp_{i,k}$ has compact support, we may choose $\eps_{i,k}>0$ such that the support of the function
		\[
		\psi_{i,k}(x)=\eps_{i,k}^{\frac{p-N}{p}}\vp_{i,k}\Big(\frac{x-x_i}{\eps_{i,k}}\Big)
		\] 
is contained in $B_r(x_i)$ for all $k\in\n$.	Since $\|\psi_{i,k}\|_p=\|\vp_{i,k}\|_p$ and $|\psi_{i,k}|_{p^*}=|\vp_{i,k}|_{p^*}$, we have that $\psi_{i,k}\in M_i(\o)$ and, since $\psi_{i,k}$ and $\psi_{j,k}$ have disjoint supports whenever $i\neq j$, the function  $\bf \psi =(\psi_{1,k},\ldots,\psi_{\ell,k})$ belongs to $\cM(\o)$. Thus,
		\begin{align*}
\mu(\o)\leq \cJ(\bf \psi)= \frac{1}{N}\sum_{i=1}^{\ell}\|\psi_{i,k}\|_p^p= \frac{1}{N}\sum_{i=1}^{\ell}\|\vp_{i,k}\|_p^p=\sum_{i=1}^{\ell}J_i(\vp_{i,k}).
		\end{align*}
Passing to the limit we get $\mu(\o)\leq\sum_{i=1}^\ell m_i.$ This proves that
$$\mu(\o)=\sum_{i=1}^\ell m_i.$$

Finally, we show that $\mu(\o)$ is not attained. Arguing by contradiction, assume there exists $\bf u=(u_1,\dots, u_{\ell})\in \cM(\o)$ such that $\mu(\o)=\cJ(\bf u)$. Then, it follows from \eqref{eq:nonexistence} that
\begin{equation} \label{eq:identity}
\sum_{i=1}^\ell m_i=\frac{1}{N}\sum_{i=1}^\ell\|t_iu_i\|_p^p=\frac{1}{N}\sum_{i=1}^\ell\|u_i\|_p^p,
\end{equation}
where $t_i\in(0,1]$ is such that $t_iu_i\in M_i(\o)$. This identity implies that $t_i=1$. Hence, $u_i\in M_i(\o)$ and, as a consequence, $m_i\leq J_i(u_i)=\frac{1}{N}\|u_i\|_p^p$ for all $i=1,\ldots,\ell$. These inequalities, together with \eqref{eq:identity}, yield $m_i=\frac{1}{N}\|u_i\|_p^p$ for all $i=1,\ldots,\ell$, i.e., $u_i$ is a least energy solution of \eqref{eq:singleequation}. Since $|u_i|$ is also a least energy solution, we may assume that $u_i\geq 0$. On the other hand, since $\bf u\in \cM(\o)$, $u_i\in M_i(\o)$ and $\beta_{ij}<0$, we get that
\begin{equation*}
\irn|u_j|^{\frac{p^*}{2}}|u_i|^{\frac{p^*}{2}}=0\quad\text{if \ }j\neq i.
\end{equation*}
This implies that $u_iu_j=0$ a.e. in $\o$. Since $\ell\geq 2$ there exists $j\neq i$ and, since $u_j\neq 0$, we have that $u_i=0$ in a set of positive measure. This contradicts the maximum principle \cite[Theorem 2.5.1]{ps}.
	\end{proof}
	\smallskip
	
\begin{proof}[Proof of Theorem \ref{thm:nonexistence}]
Apply Theorem \ref{thm:o_nonexistence} to $\o=\rn$.
\end{proof}

	\section{The symmetric variational setting}
	\label{sec:pinwheel}
	
Next we consider the system
	\begin{equation}\label{eq:system}
		\begin{cases}
			-\Delta_p u_i = |u_i|^{p^*-2}u_i+\sum\limits_{\substack{j=1\\j\neq i}}^\ell\beta|u_j|^{\frac{p^*}{2}}|u_i|^{\frac{p^*}{2}-2}u_i, \\
			u_i\in D^{1,p}_0(\o),\qquad i=1,\ldots,\ell,
		\end{cases}
	\end{equation}
where $\o$ is a domain in $\rn$, $N\geq 4$, $\frac{2N}{N+2}<p<N$, $\beta<0$.
	
	We describe the variational setting for solutions that satisfy $(P_1)$ and $(P_2)$.
	
	First, we consider the group $G:=\{g\in\cc:|g|=1\}$ of unit complex numbers acting on $\rn\equiv \cc^2\times \r^{N-4}$ by
	\begin{equation}\label{eq:groupaction}
		gx:=(gz_1,\bar{g}z_2,y) \qquad \text{for every } g\in G \text{ and } x=(z_1,z_2,y)\in \cc^2\times \r^{N-4},
	\end{equation}
where $\bar{g}$ is the complex conjugate of $g$, and we assume that $\o$ is $G$-invariant (i.e., $gx\in\o$ for all $g\in G, \ x\in\o$). For each $g\in G$ and $u\in D^{1,p}_0(\o)$ we define $gu\in D^{1,p}_0(\o)$ by $(gu)(x):=u(g^{-1}x)$. Then the map $\bf u\mapsto g\bf u:=(gu_1,\dots,gu_{\ell})$ is a linear isometry of $(D^{1,p}_0(\o))^{\ell}$ and the functional $\cJ$ is $G$-invariant, i.e., $\cJ(g\bf u)=\cJ(\bf u)$ for every $g\in G, \ \bf u\in (D^{1,p}_0(\o))^{\ell}$. The $G$-fixed point space of $(D_0^{1,p}(\o))^{\ell}$ is the subspace
	\begin{align*}
		\mathcal{D}(\o):= & \{\bf u \in (D^{1,p}_0(\o))^{\ell}:g\bf u=\bf u \text{ for all }g\in G\}\\
		= & \{\bf u \in (D^{1,p}_0(\o))^{\ell}:u_j\text{ is } G\text{-invariant for each } j=1,\dots, \ell\}.
	\end{align*}
Since $(D_0^{1,p}(\o))^{\ell}$ is reflexive and strictly convex, by the principle of symmetric criticality for Banach spaces \cite[Theorem 2.2]{ko}, the critical points of the restriction of $\cJ$ to $\mathcal{D}(\o)$ are the critical points of $\cJ$ that satisfy $(P_1)$.
	
	Next, we define an action of $\zl:=\{0,1,\ldots,\ell-1\}$, the additive group of integers modulo $\ell$, on the space $\mathcal{D}(\o)$ as follows: Let $\tau :\cc^2\to\cc^2$ be given by $\tau(z_1,z_2)=(-\overline{z}_1,\overline{z}_2)$ and, for each $j\in\zl$, define
	\begin{equation*}
		\vr_\ell^jx:=\left(\Big(\cos\frac{\pi j}{\ell}\Big)z+\Big(\sin\frac{\pi j}{\ell}\Big)\tau z,\,y\right),\qquad\text{where \ }x=(z,y)\in\cc^2\times\r^{N-4}.
	\end{equation*}
	Since $\tau z$ is orthogonal to $z$ and $|\tau z|=|z|$ we have that $\vr_\ell^j\in O(N)$. Note that $\vr_\ell^k\vr_\ell^j=\vr_\ell^{k+j}$. Furthermore, as $\tau g=g\tau$ for every $g\in G$, we have that $\vr_\ell^jg=g\vr_\ell^j$ for every $g\in G$ and $j\in\zl$. 
	
	\emph{From now on we assume that $\o$ is invariant under the action of the subgroup of $O(N)$ generated by $G\cup\{\vr_\ell^1\}$.} We denote by $\sigma^j:\{1,\ldots,\ell\}\to\{1,\ldots,\ell\}$ the permutation $\sigma^j(m):=m+j \mod\ell$, \ $j\in\zl$. As in \cite[Proposition 2.1]{cfs}, it is readily seen that the function $\vr_\ell^j:\mathcal{D}(\o)\to\mathcal{D}(\o)$ given by
	$$\vr_\ell^j\bf u(x):=(u_{\sigma^j(1)}(\vr_\ell^{-j}x),\ldots,u_{\sigma^j(\ell)}(\vr_\ell^{-j}x)),\quad\text{where \ }\bf u=(u_1,\ldots,u_\ell),$$
	is a well-defined linear isometry, and that $j\mapsto\vr_\ell^j$ is a well-defined homomorphism from $\zl$ into the group of linear isometries of $\mathcal{D}(\o)$. The functional $\cJ|_{\cD(\o)}$ is, clearly, $\zl$-invariant. So by the principle of symmetric criticality \cite[Theorem 2.2]{ko}, the critical points of the restriction of $\cJ$ to the $\zl$-fixed point space of $\cD(\o)$,
	\begin{align*}
		\mathscr{D}(\o) &:=  \cD(\o)^{\zl}=\{u\in\cD(\o):\vr_\ell^j\bf u=\bf u \text{ for all }j\in \zl\}\\
		& = \{\bf u \in D^{1,p}_0(\o):u_j \text{ is } G\text{-invariant and } u_{j+1}=u_j\circ \vr_\ell \text{ for each } j=1,\dots, \ell\},
	\end{align*}
	are critical points of $\cJ|_{\cD(\o)}$ and, hence, of $\cJ$. They are the solutions to the system \eqref{eq:system} that satisfy $(P_1)$ and $(P_2)$. We call them \emph{pinwheel solutions}. 
	
	Abusing notation, we shall write
	$$\cJ:=\cJ|_{\mathscr{D}(\o)}:\mathscr{D}(\o)\to\r.$$
	Note that, if $\bf u \in\mathscr{D}(\o)$ and $\bf u\neq \bf 0$, then every component of $\bf u$ is nontrivial. Hence, the fully nontrivial critical points of $\cJ$ belong to the Nehari manifold
	\begin{align*}
		\cN(\o):= \{\bf u \in \mathscr{D}(\o):\bf u\neq \bf 0 \text{ and } \cJ'(\bf u)\bf u\neq 0\}. 
	\end{align*}
Note also that
	$$\cJ'(\bf u)\bf v=\sum_{i=1}^\ell\partial_i \cJ (\bf u)v_i= \ell \partial_j\cJ(\bf u)v_j \quad \text{ for all }\bf u, \bf v\in\mathscr{D}(\o)\text{ and } j=1,\dots, \ell. $$
Therefore $\cN(\o)=\cM(\o)\cap \mathscr{D}(\o)$, where $\cM(\o)$ is as in \eqref{M(o)}. Set
	\begin{equation*}
		c(\o):= \inf_{\bf u\in \cN(\o)}\cJ(\bf u).
	\end{equation*}
Then $c(\o)\geq\mu(\o)$ and it follows from Theorem \ref{thm:o_nonexistence} and equation \eqref{eq:m_i} that
\begin{equation}\label{eq:Sp}
c(\o)\geq \frac{\ell}{N}S_p^\frac{N}{p}.
\end{equation}	

	\begin{lemma}\label{lem:nehari}
 $\cN(\o)\neq \emptyset$, and it is a closed $\cC^1$-submanifold of $\mathscr{D}(\o)$ and a natural constraint for $\cJ$. 
	\end{lemma}
	
	\begin{proof}
The proof is easy; see \cite[Lemma 2.3]{cfs}.
	\end{proof}
	
	The proof of the following lemma is identical to that of \cite[Proposition 2.5$(i)$]{cfs}. We give the details for the benefit of the reader.
	
	\begin{lemma} \label{lem:same infimum}
		If $\o\cap(\{0\}\times\r^{N-4})\neq\emptyset$, then  $c(\o)=c(\rn)$.
	\end{lemma}
	
	\begin{proof}
		As $\cN(\o)\subset\cN(\rn)$ via trivial extension, we have that $c(\o)\geq c(\rn)$. To prove the opposite inequality we chose a sequence $\bf\vp_k=(\vp_{1,k},\ldots,\vp_{\ell,k})\in\cN(\rn)$ such that $\vp_{j,k}\in\cC^\infty_c(\rn)$ and $\cJ(\bf\vp_k)\to c(\rn)$. Let $\xi\in\o\cap(\{0\}\times\r^{N-4})$ and $r>0$ be such that $B_r(\xi)\subset\o$, where $B_r(\xi)$ is the ball centered at $\xi$ of radius $r$.  Fix $\eps_k>0$ such that $\eps_kz\in B_r(0)$ for every $z\in\supp(\vp_{1,k})$. Since $\vp_{j+1,k}=\vp_{1,k}\circ\vr_\ell^j$ for every $j=0,\ldots,\ell-1$, we have that $\eps_kx\in B_r(0)$ for every $x\in\supp(\vp_{j,k})$ and $j=1,\ldots,\ell$. Define
		$$\widetilde{\vp}_{j,k}(x):=\eps_k^\frac{p-N}{p}\vp_{j,k}\Big(\frac{x-\xi}{\eps_k}\Big).$$
		Then $\widetilde{\vp}_{j,k}\in\cC^\infty_c(\o)$ and, since $g\xi=\xi$ for every $g\in G$ and $\vr_\ell^j\xi=\xi$, we have that $\widetilde{\bf\vp}_k=(\widetilde{\vp}_{1,k},\ldots,\widetilde{\vp}_{\ell,k})\in\mathscr{D}(\o)$. Furthermore, as
		$$\|\widetilde{\vp}_{j,k}\|^p_p=\|\vp_{j,k}\|^p_p\qquad\text{and}\qquad\io|\widetilde{\vp}_{j,k}|^\frac{p^*}{2}|\widetilde{\vp}_{i,k}|^\frac{p^*}{2}=\irn|\vp_{j,k}|^\frac{p^*}{2}|\vp_{i,k}|^\frac{p^*}{2},$$
		we have that $\widetilde{\bf\vp}_k\in\cN(\o)$ and $\cJ(\widetilde{\bf\vp}_k)=\cJ(\bf\vp_k)\to c(\rn)$. This shows that $c(\o)\leq c(\rn)$ and completes the proof.
	\end{proof}
	
	\begin{remark}
		\emph{When $p=2$ more can be said. Namely, as shown in \cite[Proposition 2.5$(ii)$]{cfs}, if $\o\cap(\{0\}\times\r^{N-4})\neq\emptyset$ and $\rn\smallsetminus\o$ has nonempty interior, then the system \eqref{eq:system} does not have a least energy pinwheel solution. This follows from the unique continuation property for systems proved in \cite{chs}, which is not available for $p\neq 2$.}
	\end{remark}
	
	\section{Existence of a pinwheel solution}
	\label{sec:existence}
	
	Our aim is to analyse the behavior of minimizing sequences for $\cJ$ on $\cN(\o)$ when $\o$ is the unit ball. The existence of a pinwheel solution to the system \eqref{eq:rn_symmetric_system} will follow from this analysis. We need the following lemma.
	
	\begin{lemma}\label{lem:G}
		Let $G$ act on $\rn\equiv\cc^2\times\r^{N-4}$ as in \eqref{eq:G}. Then any given sequences $(\eps_k)$ in $(0,\infty)$ and $(\zeta_{k})$ in $\rn$ contain subsequences that satisfy one of the following statements:
		\begin{enumerate}
			\item[$(i)$] either there exist $\eta_k\in\{0\}\times\r^{N-4}$ and $C_1>0$ such that $\eps_{k}^{-1}|\zeta_{k}-\eta_k|<C_1$ for all $k\in\n$,
			\item[$(ii)$] or, for each $m\in\n$, there exist $g_1,\ldots,g_m\in G$ such that $\eps_{k}^{-1}|g_i\zeta_{k}-g_j\zeta_{k}|\rightarrow\infty$ for any $i\neq j$.
		\end{enumerate}
	\end{lemma}
	\begin{proof}
		See \cite[Lemma 3.1]{cfs}.
	\end{proof}	
	
	\begin{theorem} \label{thm:minimizing}
		Let $\b$ be the unit ball centered at the origin in $\rn$ and $(\bf u_k)$ be a sequence in $\cN(\b)$ such that $\cJ(\bf u_k)\to c(\b)$. Then, after passing to a subsequence, one of the following three statements holds true:
		\begin{itemize}
			\item[$(I)$] $(u_k)$ converges strongly in $\mathscr{D}(\b)$ to a least energy pinwheel solution of the system \eqref{eq:system} in the ball $\b$.
			\item[$(II)$] There exist sequences $(\eps_k)$ in $(0,\infty)$ and $(\xi_k)$ in $\partial\b\cap(\{0\}\times\r^{N-4})$ such that $\eps_k\to 0$ and $\xi_k\to\xi$, and a least energy pinwheel solution $\bf w=(w_1,\ldots,w_\ell)$ to the system \eqref{eq:system} in the half-space
			$$\mathbb{H}:=\{x\in\rn:\xi\cdot x<0\}$$
			such that \ $\lim_{k\to\infty}\|\bf u_k - \widehat{\bf w}_k\|=0$, \ where $\widehat{\bf w}_k=(\widehat{w}_{1,k},\ldots, \widehat{w}_{\ell,k})$ is given by
			\begin{equation*}
				\widehat{w}_{i,k}(x) =\eps_{k}^{\frac{p-N}{p}}w_i\Big(\frac{x-\xi_k}{\eps_{k}}\Big)\qquad \text{for every \ } k\in\n \text{ \ and \ }i=1,\ldots,\ell.
			\end{equation*}
			\item[$(III)$] There exist sequences $(\eps_k)$ in $(0,\infty)$ and $(\xi_k)$ in $\b\cap(\{0\}\times\r^{N-4})$ and a least energy pinwheel solution $\bf w=(w_1,\ldots,w_\ell)$ to the system \eqref{eq:rn_symmetric_system} in $\rn$, such that
			\begin{align*}
				\eps_k^{-1}\dist(\xi_k,\partial\b)\to\infty,\qquad\text{and}\qquad\lim_{k\to\infty}\|\bf u_k - \widehat{\bf w}_k\|=0,
			\end{align*}
			where $\widehat{\bf w}_k=(\widehat{w}_{1,k},\ldots, \widehat{w}_{\ell,k})$ is defined as above.
		\end{itemize}
	\end{theorem}
	
	\begin{proof}
Let $\bf u_k=(u_{1,k},\ldots,u_{\ell,k})\in\cN(\b)$ be such that $\cJ(\bf u_k)\rightarrow c(\b)$. By Ekeland's variational principle, we may assume that $\cJ'(\bf u_k)\rightarrow 0$ in $[(D_0^{1,p}(\b))^{\ell}]'$. Since $\cJ(\bf u_k)=\frac{1}{N}\|\bf u_k\|_p^p$, \ $(\bf u_k)$ is bounded in $(D_0^{1,p}(\b))^{\ell}$. Hence, it contains a subsequence that converges weakly to some $\bf u=(u_1,\ldots,u_\ell)$ in $\mathscr{D}(\b)$.  Following an argument similar to that presented in \cite[Appendix A]{cl}, we see that $\bf u$ is a solution of the system \eqref{eq:system}.

If $\bf u\neq \bf 0$ then $\bf u \in \cN(\b)$ and 
		\[
		c(\b) \leq \cJ(\bf u) =\frac{1}{N}\sum_{i=1}^{\ell}\|u_{i}\|_p^p\leq \lim_{k\rightarrow \infty}\frac{1}{N}\sum_{i=1}^{\ell}\|u_{i,k}\|_p^p=c(\b).
		\]  
Hence, $c(\b)=\cJ(\bf u)$ and $\bf u$ is a least energy solution of \eqref{eq:system} in $\b$. Furthermore, since $\| u_{i,k}\|_p\rightarrow \|u_k\|_p$ and $D_0^{1,p}(\b)$ is a uniformly convex Banach space, $u_{i,k}\rightarrow u_k$ strongly in $D_0^{1,p}(\b)$ for all $i=1,\dots, \ell$; see \cite[Proposition 3.32]{brezis}. This shows that statement $(I)$ holds true if $\bf u\neq \bf 0$.
		
Now assume that $\bf u=\bf 0$. Fix $0<2\delta<S_p^{\frac{N}{p}}$.  Since $\beta<0$, using \eqref{eq:Sp} we see that 
$$\frac{\ell}{N}|u_{1,k}|_{p^*}^{p^*}\geq\cJ(\bf u_k)=c(\b)+o(1)\geq \frac{\ell}{N}S_p^\frac{N}{p}+o(1)>\frac{\ell}{N}\delta,$$
for large enough $k$. Hence, there exist bounded sequences $(\eps_k)$ and $(\zeta_k)$ in $(0,\infty)$ and $\rn$, respectively, such that
		\begin{equation*}
		\delta=\sup_{x\in \rn}\int_{B_{\eps_k}(x)}|u_{1,k}|^{p^*}=\int_{B_{\eps_k}(\zeta_k)}|u_{1,k}|^{p^*}\quad \text{for all }\; k\in \mathbb{N}.
		\end{equation*}
Applying Lemma \ref{lem:G} to the sequences $(\eps_k)$ and $(\zeta_k)$ we have two options. Let us see that option $(ii)$ yields a contradiction. Indeed,  suppose that for each $m\in \mathbb{N}$ there exist $g_1,\dots, g_m\in G$
		such that $\eps_k^{-1}|g_i\zeta_k-g_j\zeta_k|\rightarrow\infty$, if $i\neq j$. Then, for $k$ large enough,
		\[
		B_{\eps_k}(g_i\zeta_k)\cap B_{\eps_k}(g_j\zeta_k)=\emptyset \quad\text{if \ } i\neq j.
		\]
		Since $u_{1,k}$ is $G$-invariant, we have
		\[
		m\delta=m\int_{B_{\eps_k}(\zeta_k)}|u_{1,k}|^{p^*}= \sum_{i=1}^m\int_{B_{\eps_k}(g_i\zeta_k)}|u_{1,k}|^{p^*}\leq |u_{1,k}|_{p^*}^{p^*} \quad \text{for all  } m\in \mathbb{N}.
		\]
This is a contradiction, because $(u_{1,k})$ is bounded in $D_0^{1,p}(\b)$. Therefore, option $(i)$ in Lemma \ref{lem:G} must hold true, i.e., after passing to a subsequence, there exist $\eta_k\in \{0\}\times \r^{N-4}$ and $C_1>0$ such that $\eps_k^{-1}|\zeta_k-\eta_k|<C_1$ for all $k\in \mathbb{N}$. Thus, $B_{\eps_k}(\zeta_k)\subset B_{(C_1+1)\eps_k}(\eta_k)$ and 
		\begin{equation}\label{eq:auxiliar2}
			0<\delta=\int_{B_{\eps_k}(\zeta_k)}|u_{1,k}|^{p^*}\leq \int_{B_{(C_1+1)\eps_k}(\eta_k)}|u_{1,k}|^{p^*}.
		\end{equation}
		Passing to a subsequence we have
		\[
		d_k=\eps_k^{-1}\operatorname{dist}(\eta_k,\b)\rightarrow d\in[0,\infty].
		\]
		If $0\leq d<\infty$, for each $k$ we choose $\xi_k\in \partial \b \cap (\{0\}\times \r^{N-4})$ such that $|\eta_k-\xi_k|=\operatorname{dist}(\eta_k,\b)$. On the other hand, if $d=\infty$, we set $\xi_k=\eta_k$. Then, inequality \eqref{eq:auxiliar2} implies that $\operatorname{dist}(\xi_k,\b)\leq (C_1+1)\eps_k$ and, since $d=\infty$, necessarily $\xi_k\in \b$. Summing up, there are two possibilities:
		\begin{itemize}
			\item[(1)] either $\xi_k\in \partial\b \cap (\{0\}\times \r^{N-4})$,
			\item[(2)] or $\xi_k\in \b\cap (\{0\}\times \r^{N-4})$ and $\eps_k^{-1}\operatorname{dist}(\xi_k,\partial\b)\rightarrow \infty$.
		\end{itemize}
		In both cases there is a constant $C_0>0$ such that $\eps_k^{-1}|\zeta_k-\xi_k|<C_0$. Thus, $B_{\eps_k}(\zeta_k)\subset B_{(C_0+1)\eps_k}(\xi_k)$ for $k$ large enough and as a consequence 
		\begin{equation}\label{eq:auxiliar3}
0<\delta=\sup_{x\in \rn}\int_{B_{\eps_k}(x)}|u_{1,k}|^{p^*}=\int_{B_{\eps_k}(\zeta_k)}|u_{1,k}|^{p^*}\leq \int_{B_{(C_0+1)\eps_k}(\xi_k)}|u_{1,k}|^{p^*}.
		\end{equation}
		Let $\Omega_k:=\{x\in \rn : \eps_kx+\xi_k\in \b\}$ and consider the function $\bf w_k=(w_{1,k},\dots, w_{\ell,k})$ defined by 
		\[
		w_{i,k}(x):=\eps_k^{\frac{N-p}{p}}u_{i,k}(\eps_kx+\xi_k)\; \text{ if }\; x\in \Omega_k, \qquad w_{i,k}(x):=0 \;\text{ otherwise.}
		\]
		Since $\xi_k\in \{0\}\times \r^{N-4}$ we have that $g\xi_k=\xi_k$ for every $g\in G$ and $\rho_{\ell}\xi_k=\xi_k$. Therefore $\bf w_k\in \mathscr{D}(\rn)$ for every $k$ and, as $\|w_{i,k}\|_p=\|u_{i,k}\|_p$ for $i=1,\dots,\ell$, after passing to a subsequence we have that
		\[
		\bf w_k\rightharpoonup \bf w \text{  weakly in  } \mathscr{D}(\rn),\quad \bf w_k\rightarrow \bf w \text{ in  }(L_{loc}^p(\rn))^{\ell},\quad \text{and  } \bf w_k\rightarrow \bf w \text{  a.e. in  } \rn.
		\]
		Additionally, \eqref{eq:auxiliar3} yields
		\begin{equation}\label{eq:auxiliar4}
			\delta= \sup_{x\in \rn}\int_{B_1(x)}|w_{1,k}|^{p^*}\leq \int_{B_{C_0+1}(0)}|w_{1,k}|^{p^*}.
		\end{equation}
Next we show that $\bf w\neq 0$. 

Arguing by contradiction, assume that $\bf w=0$. Let $\vp\in \cC_c^{\infty}(\rn)$, $\vp\geq 0$, and set $\vp_k(x):= \vp\big(\frac{x-\xi_k}{\eps_k}\big)$. Then, $(\vp_k^p u_{1,k})$ is a bounded sequence in $D_0^{1,p}(\b)$. Since $\cJ'(\bf u_k)\rightarrow 0$ in $[(D_0^{1,p}(\b))^{\ell}]'$, a direct computation shows that
\begin{align} \label{eq:w1}
o(1)&=\partial_1\cJ(\bf u_k)[\vp_k^pu_{1,k}]\\
&=\irn|\nabla w_{1,k}|^{p-2}\nabla w_{1,k}\cdot \nabla (\vp^pw_{1,k})-\irn|w_{1,k}|^{p^*}\vp^p-\beta\sum_{\overset{j=1}{j\neq 1}}\irn |w_{j,k}|^{\frac{p^*}{2}}|w_{1,k}|^{\frac{p^*}{2}}\vp^p.\nonumber
\end{align}
Now, since \ $	|\nabla w_{1,k}|^{p-2}\nabla w_{1,k}\cdot \nabla (\vp^pw_{1,k})= |\vp\nabla w_{1,k}|^p+p\vp^{p-1}w_{1,k}|\nabla w_{1,k}|^{p-2}\nabla w_{1,k}\cdot \nabla \vp$ \ and
\begin{align*}
\irn \left|\vp^{p-1}w_{1,k}|\nabla w_{1,k}|^{p-2}\nabla w_{1,k}\cdot \nabla \vp\right|&\leq \irn \vp^{p-1}|\nabla w_{1,k}|^{p-1}|w_{1,k}\nabla \vp|\\
&\leq C\|w_{1,k}\|_p^{p-1}\Big(\int_{\supp(\vp)}|w_{1,k}|^p\Big)^{1/p}=o(1),
\end{align*}
(because $w_{1,k}\rightarrow 0$ in $L_{loc}^p(\rn)$), we have that
\begin{equation}\label{eq:w2}
\irn|\nabla w_{1,k}|^{p-2}\nabla w_{1,k}\cdot \nabla (\vp^pw_{1,k})=\irn|\vp\nabla w_{1,k}|^p+o(1).
\end{equation}
Using the mean value theorem we see that
		\begin{align*}
\Big||\nabla(\vp w_{1,k})|^p-|\vp\nabla w_{1,k}|^p\Big| & = \Big||\vp\nabla w_{1,k}+w_{1,k}\nabla \vp|^p-|\vp\nabla w_{1,k}|^p\Big| \\
			&\leq p\Big(|w_{1,k}||\nabla \vp|+|\vp||\nabla w_{1,k}|\Big)^{p-1}|w_{1,k}\nabla \vp|.
		\end{align*}
So integrating and using $w_{1,k}\rightarrow 0$ in $L_{loc}^p(\rn)$ again we get
		\begin{align*}
\left|\irn |\nabla(\vp w_{1,k})|^p-\irn|\vp\nabla w_{1,k}|^p \right| & \leq C\int_{\supp(\vp)}\Big(|w_{1,k}|+|\nabla w_{k,1}|\Big)^{p-1}|w_{1,k}|\\
			&\leq C\left(\int_{\supp(\vp)}\Big||w_{1,k}|+|\nabla w_{1,k}|\Big|^{p}\right)^{\frac{p-1}{p}}\left(\int_{\supp(\vp)}|w_{1,k}|^p\right)^{\frac{1}{p}}=o(1).
		\end{align*}
Thus,
\begin{equation}\label{eq:w3}
\irn |\nabla(\vp w_{1,k})|^p=\irn|\vp\nabla w_{1,k}|^p+o(1).
\end{equation}
Now assume that, in addition, $\supp(\vp)\subset B_1(x)$ for some $x\in\rn$. Then, since $\beta<0$, from  \eqref{eq:w1}, \eqref{eq:w2}, \eqref{eq:w3} and our choice of $\delta$ we get
\begin{align*}
\irn|\nabla(\vp w_{1,k})|^p	&\leq \irn|w_{1,k}|^{p^*}\vp^{p}+o(1) \leq \Big(\int_{B_1(x)}|w_{1,k}|^{p*}\Big)^{\frac{p^*-p}{p^*}}\Big(\irn|\vp w_{1,k}|^{p^*}\Big)^{\frac{p}{p^*}}+o(1)\\
&\leq \delta S_p^{-1}\irn |\nabla (\vp w_{1,k})|^p+o(1)<\Big(\frac{1}{2}\Big)^{\frac{N}{p}}\irn |\nabla (\vp w_{1,k})|^p+o(1).
\end{align*} 
It follows that $\|\vp w_{1,k}\|_p=o(1)$ and, hence, that $|\vp w_{1,k}|_{p^*}=o(1)$ for every $\vp\in\cC^\infty_c(B_1(x))$ such that $\vp\geq 0$ and any $x\in\rn$. This implies that $w_{1,k}\rightarrow 0$ in $L^{p^*}_{loc}(\rn)$, contradicting \eqref{eq:auxiliar4}. Therefore, $\bf w\neq \bf 0$, as claimed.
		
	Note that, as $u_{k,1}\rightharpoonup 0$ and $w_{k,1}\rightharpoonup w_1\neq 0$ weakly in $D^{1,p}(\rn)$, we have $\eps_k\rightarrow 0$. 
		
	Let us now analyze the alternatives (1) and (2). 
	
	If (1) holds true, passing to a subsequence we have that $\xi_k\rightarrow \xi\in \partial \b \cap (\{0\}\times \r^{N-4})$. Then $\mathbb{H}:=\{x\in \rn \,:\, \xi\cdot x<0\}$ is invariant under the action of the group generated by $G\cup \{\rho_{\ell}\}$) and $\bf w\in \mathscr{D}(\mathbb{H})$. If $\varphi\in \cC_c^{\infty}(\mathbb{H})$, then $\supp (\varphi) \subset \Omega_k$ for $k$ large enough. Hence, the support of the function $\varphi_k(x):=\eps_k^{\frac{p-N}{p}}\varphi\big(\frac{x-\xi_k}{\eps_k}\big)$ is contained in $\b$ and performing a change of variables we obtain
		\begin{align*}
			\partial_1\cJ(\bf w_k)\varphi&=\irn |\nabla w_{k,1}|^{p-2}\nabla w_{1,k}\cdot \nabla \varphi-\irn |w_{1,k}|^{p^*-2}w_{1,k}\varphi-\beta\sum_{\overset{j=1}{j\neq 1}}\irn|w_{j,k}|^{\frac{p^*}{2}}|w_{1,k}|^{\frac{p^*}{2}-2}w_{1,k}\varphi\\
			&= \partial_1\cJ(\bf u_k)\varphi_k=o(1).
		\end{align*} 
		As $\bf w_k\rightharpoonup\bf w$ weakly in $\mathscr{D}(\rn)$, following an argument similar to that presented in \cite[Appendix A]{cl} we derive that
		\[
		\partial_i\cJ(\bf w)\varphi=\lim_{k\to\infty}\partial_i\cJ(\bf w_k)\varphi=0\quad \text{for all }\; \varphi\in \cC_c^{\infty}(\mathbb{H}),\;\; i=1,\dots, \ell.
		\]
		Therefore, $\bf w\in \cN(\mathbb{H})$, and using Lemma \ref{lem:same infimum}, we have
		\[
		c(\mathbb{H})\leq \cJ(\bf w)=\frac{1}{N}\sum_{i=1}^{\ell}\|\bf w_{i}\|_p^p\leq\liminf_{k\to \infty} \frac{1}{N}\sum_{i=1}^{\ell}\|\bf w_{i,k}\|_p^p=\liminf_{k\to \infty} \frac{1}{N}\sum_{i=1}^{\ell}\|\bf u_{i,k}\|_p^p=c(\b)= c(\mathbb{H}).
		\]
		This shows that $\bf w$ is a least energy pinwheel solution of \eqref{eq:system} in $\mathbb{H}$ and that $\bf w_k\rightarrow \bf w$ strongly in $\mathscr{D}(\mathbb{H})$. Setting $ \widehat{\bf w}_k=(\widehat{w}_{1,k},\dots, \widehat{w}_{\ell,k})$ with
		\[
	\what w_{i,k}(x)=\eps_k^{\frac{p-N}{p}}w_i\Big(\frac{x-\xi_k}{\eps_k}\Big)
		\]
		we get that $\|\bf u_k-\widehat{\bf w}_k\|_p=\|\bf w_k-\bf w\|_p=o(1)$. Therefore, if (1) is true, statement (II) is necessarily true.
		
		Assume now alternative (2), i.e., $\xi_k\in \b \cap (\{0\}\times \r^{N-4})$ and $\eps_k^{-1}\operatorname{dist}(\xi_k,\partial\b)\rightarrow \infty$. Then, for any $\varphi\in \cC_c^{\infty}(\rn)$, we have that $\supp (\varphi)\subset \Omega_k$ if $k$ is large enough. Arguing as before we see that
		\[
		\partial_i\cJ(\bf w)\varphi=\lim_{k\to\infty}\partial_i\cJ(\bf w_k)\varphi=0\quad \text{for all }\; \varphi\in \cC_c^{\infty}(\mathbb{R}^N),\;\; i=1,\dots, \ell.
		\]
		Thus, $\bf w \in \cN(\rn)$ and using Lemma \ref{lem:same infimum} we obtain
		\[
		c(\mathbb{R}^N)\leq \cJ(\bf w)=\frac{1}{N}\sum_{i=1}^{\ell}\|\bf w_{i}\|_p^p\leq\liminf_{k\to \infty} \frac{1}{N}\sum_{i=1}^{\ell}\|\bf w_{k,i}\|_p^p=\liminf_{k\to \infty} \frac{1}{N}\sum_{i=1}^{\ell}\|\bf u_{k,i}\|_p^p=c(\b)= c(\mathbb{R}^N).
		\]
		This shows that $\bf w$ is a least energy pinwheel solution of \eqref{eq:system} and that $\bf w_k\rightarrow \bf w$ strongly in $\mathscr{D}(\rn)$. So, setting $\widehat{\bf w}_k=(\widehat{w}_{1,k},\dots, \widehat{w}_{\ell,k})$ with $\what w_{i,k}$ as above, we get that $\|\bf u_k-\widehat{\bf w}_k\|_p=\|\bf w_k-\bf w\|_p=o(1)$. Therefore, if (2) is true, statement (III) is necessarily true.
		
		The proof is now complete.
	\end{proof}
\smallskip	

	\begin{proof}[Proof of Theorem \ref{thm:existence}]
		Since $\cN(\b)\neq \emptyset$ by Lemma \ref{lem:nehari}, there is a sequence $(\bf u_k)$ in $\cN(\b)$ such that $\cJ(\bf u_k)\to c(\b)$. It follows from Theorem \ref{thm:minimizing} that the system \eqref{eq:system} has a least energy pinwheel solution either in $\b$, or in $\mathbb{H}$, or in $\rn$. From Lemma \ref{lem:same infimum} we get that the energy of any one of them is $c(\rn)$. Therefore, if the system has a least energy pinwheel solution either in $\b$ or in $\mathbb{H}$, its trivial extension to $\rn$ is a least energy pinwheel solution to the system \eqref{eq:system} in $\rn$.  This proves that, in all three cases, there exists a least energy pinwheel solution $\bf v=(v_1,\ldots,v_\ell)$ to the system \eqref{eq:rn_symmetric_system}. Then, $\bf u=(|v_1|,\ldots,|v_\ell|)\in \cN(\rn)$ and $\cJ(\bf u)=c(\rn)$. So $\bf u$ is a least energy pinwheel solution to \eqref{eq:rn_symmetric_system} with nonnegative components.
	\end{proof}
	
\begin{remark} \label{rem:unique continuation}
\emph{When $p=2$ the statements $(I)$ and $(II)$ of Theorem \ref{thm:minimizing} are discarded by the unique continuation property for systems that was proved in \cite{chs}, leaving $(III)$ as the only option; see \cite[Theorem 3.2]{cfs}. This immediately yields the existence of a pinwheel solution of the semilinear system in $\rn$. However, even for $p=2$, it is not clear whether the components can vanish in some open subset of $\rn$. The result in \cite{chs} guarantees only that not all components vanish in the same open subset.}
\end{remark}
	
	\section{Infinitely many solutions to the quasilinear equation}
	\label{sec:nodal}
	
	To prove Theorem \ref{thm:multiplicity} we use the following result from \cite{cl}.
	
\begin{theorem}\label{thm:cl}
Let $\Gamma$ be a closed subgroup of $O(N)$ and $\phi:\Gamma\to\mathbb{Z}_2:=\{1,-1\}$ be a continuous homomorphism of groups with the following properties: 
\begin{itemize}
\item[$(S1)$]For each $x \in \rn$, either $\#\Gamma x=\infty$ or $\#\Gamma x=1$, where $\Gamma x:=\{\gamma x:\gamma\in\Gamma\}$ is the $\Gamma$-orbit of $x$ and $\#\Gamma x$ is its cardinality.
\item[$(S2)$]$\phi:\Gamma\to\mathbb{Z}_2$ is surjective.
\item[$(S3)$]There exists $\xi\in\rn$ such that $\{\gamma\in\Gamma:\gamma\xi=\xi\}\subset\ker\phi$.
\end{itemize}
Then, the problem \eqref{eq:equation} has a nontrivial solution $w$ that satisfies
\begin{equation*}
w(\gamma x)=\phi(\gamma)w(x)\quad\text{for all \ }\gamma\in\Gamma, \ x\in\rn.
\end{equation*}
In particular, $w$ is nonradial and changes sign.
\end{theorem}

\begin{proof}
See \cite[Theorem 3.1]{cl}.
\end{proof}
\smallskip
	
	\begin{proof}[Proof of Theorem \ref{thm:multiplicity}]
		Let $G$ act on $\rn$ as in \eqref{eq:G} and $\varrho_{\ell}$ be as defined in \eqref{eq:rho}. Let $\Gamma_{\ell}$ be the subgroup of $O(N)$ generated by $G\cup \{\varrho_{\ell}\}$ and $\phi_{\ell}:\Gamma_{\ell}\rightarrow \mathbb{Z}_2:=\{-1,1\}$ be the homomorphism of groups given by $\phi_{\ell}(g)=1$ if $g\in G$ and $\phi_{\ell}(\varrho_{\ell})=-1$. Since $\varrho_{\ell}$ has order $2\ell$, this homomorphism is well defined and it is clearly surjective. The $\Gamma_{\ell}$-orbit $\Gamma_{\ell} x:=\{ \gamma x\, :\, \gamma\in \Gamma_{\ell}\}$ of a point $x=(z,y)\in \mathbb{C}^2\times \r^{N-4}\equiv\rn$ satisfies
\begin{equation*}
\#\Gamma_\ell x=
\begin{cases}
\infty &\text{if \ }z\neq 0,\\
1&\text{if \ }z=0.
\end{cases}
\end{equation*}		
Furthermore, if $e_1=(1,0,\dots ,0)$, then $\{\gamma \in \Gamma_{\ell}\, :\, \gamma e_1=e_1\}=\{1\}\subset \ker \phi_{\ell}$. This proves that $(S1)-(S3)$ hold true. So, by Theorem \ref{thm:cl}, there exists a nonradial sign-changing solution $w_{\ell}$ of equation \eqref{eq:equation} such that
		\[
		w_{\ell}(gx)=w_{\ell}(x)\quad \text{and }\;\; w_{\ell}(\varrho_{\ell}x)=-w_{\ell}(x)\quad \text{for all  }\; g\in G, \ x\in \rn.
		\]
To prove that problem \eqref{eq:equation} has infinitely many nonradial sign-changing solutions we show next that
$$w_1,w_2,\ldots,w_{2^k},\ldots$$
are all distinct. To this end if suffices to show that, if $\ell=nm$ with $n$ even and $u$ and $v$ are nontrivial functions that satisfy
		\[
		u(\varrho_{\ell}x)=-u(x) \;\;\; \text{and   }\; v(\varrho_{m}x)=-v(x)\;\;\; \text{ for every } x\in \rn,
		\]
then $u\neq v$. Indeed, arguing by contradiction, assume that $u=v$ and choose $x_0\in \rn$ such that $u(x_0)=v(x_0)\neq 0$. Since $\varrho_{\ell}^n=\varrho_{m}$ and $n$ is even, we have that
		\[
		u(\varrho_{m}x_0)=u(\varrho_{\ell}^nx_0)=u(x_0)=v(x_0)=-v(\varrho_{m} x_0)= -u(\varrho_{m} x_0),
		\]
		which is a contradiction.
	\end{proof}
	
For the Yamabe problem \eqref{eq:equation} with $p=2$ these solutions were recently obtained in \cite[Theorem 1.3]{cfs}.


	\bigskip

	\begin{flushleft}
		\textbf{Mónica Clapp} and \textbf{Víctor A. Vicente-Benítez}\\
		Instituto de Matemáticas\\
		Universidad Nacional Autónoma de México \\
		Campus Juriquilla\\
		76230 Querétaro, Qro., Mexico\\
		\texttt{monica.clapp@im.unam.mx} \\
		\texttt{va.vicentebenitez@im.unam.mx} 
		
	\end{flushleft}
	
\end{document}